\documentclass[11pt]{amsart}

\usepackage{amsmath,amssymb,amsthm,mathtools}
\usepackage{enumitem}
\usepackage[hidelinks]{hyperref}
\usepackage[nameinlink,noabbrev]{cleveref}

\DeclareMathOperator{\lcc}{lcc}
\DeclareMathOperator{\val}{val}

\newtheorem{theorem}{Theorem}[section]
\newtheorem{proposition}[theorem]{Proposition}
\newtheorem{lemma}[theorem]{Lemma}
\newtheorem{corollary}[theorem]{Corollary}

\theoremstyle{definition}

\theoremstyle{remark}

\crefname{conjecture}{conjecture}{conjectures}
\Crefname{conjecture}{Conjecture}{Conjectures}
\crefname{question}{question}{questions}
\Crefname{question}{Question}{Questions}

\title[Local clique covers and chromatic number]
{Local Clique Covers and Chromatic Number}
\author{Akbar Davoodi}
\date{}
\hypersetup{
pdftitle={Local Clique Covers and Chromatic Number},
pdfauthor={Akbar Davoodi},
pdfsubject={The chromatic bound and structural properties of local clique covers}
}

\begin{document}

\begin{abstract}
	The local clique cover number $\lcc(G)$ is the minimum valency of an edge-clique cover of $G$. We prove the conjectured inequality
	$\lcc(G)+\chi(G)\le |V(G)|+1$ for every finite simple graph. The proof gives an independent set improvement of an endpoint-cover
	estimate and applies the resulting construction to an induced four-vertex path. We further prove that the cover can be chosen so
	that every nonuniversal vertex has valency at most $|V(G)|-\chi(G)$. Consequently, every graph attaining equality has a universal vertex. The stronger statement follows by reducing a counterexample of minimum order to a prime double-critical graph and refining the induced path extension. We also show that an induced matching of size $m\ge2$ yields $\lcc(G)+\chi(G)\le |V(G)|+3-m$, which improves the general bound when $m\ge3$, and determine the restrictions imposed by equality on deletion of an induced $2K_2$.
\end{abstract}

\maketitle

\section{Introduction}\label{sec:introduction}
An \emph{edge-clique cover} of a graph $G$ is a family of cliques whose members together contain every edge of $G$. If $\mathcal C$ is such a cover, the \emph{valency} of a vertex $v$ in $\mathcal C$ is the number of members of $\mathcal C$ that contain $v$. The valency of the cover is the maximum of these vertex valencies. The \emph{local clique cover number} $\lcc(G)$ is the minimum valency of an edge-clique cover of $G$. Thus the parameter depends on the distribution of the covering cliques among the vertices. In particular, replacing a covering clique by a larger maximal clique may increase the valencies at the added vertices.

Classical set intersection representations go back at least to Szpilrajn-Marczewski and to Erd\H{o}s, Goodman, and P\'osa
\cite{SzpilrajnMarczewski1945,ErdosGoodmanPosa1966}. Javadi, Maleki, and Omoomi \cite{JavadiMalekiOmoomi2016} observed that $\lcc(G)$ is exactly the least nonnegative integer $q$ for which there are sets $(A_v:v\in V(G))$ with $|A_v|\le q$ for every $v\in V(G)$, such that, for all distinct $u,v\in V(G)$,
\[
uv\in E(G)\quad\Longleftrightarrow\quad A_u\cap A_v\ne\varnothing.
\]
This relates local clique covering to graph coloring and hypergraph representation. A related parameter is the \emph{sigma clique cover number}, also called the edge clique cover sum, which minimizes the sum of the orders of the covering cliques, or equivalently the total vertex valency \cite{DavoodiJavadiOmoomi2016Sum}. Its edge-partition analogue is studied in \cite{DavoodiJavadiOmoomi2016PBD}. For complete multipartite graphs with a fixed common part size $d\ge2$, Davoodi, Gerbner, Methuku, and Vizer \cite{DavoodiGerbnerMethukuVizer2020} obtained asymptotically sharp bounds for the sigma clique cover number as the number of parts tends to infinity. The local parameter instead requires an upper bound on the valency at each vertex simultaneously.

Bujt\'as, Davoodi, Gy\H{o}ri, and Tuza \cite{BujtasDavoodiGyoriTuza2020} recorded the conjecture, attributed there to Ramin Javadi in 2012, that
\begin{equation}\label{eq:chromatic-bound}
	\lcc(G)+\chi(G)\le |V(G)|+1.
\end{equation}
They proved it for claw-free graphs and established the general clique-number bound
$\lcc(G)+\omega(G)\le |V(G)|+1$,
which implies the chromatic inequality for perfect graphs. General degree bounds and the equality case $\lcc(G)=\Delta(G)$ were obtained in \cite{JavadiMalekiOmoomi2016}. Our first result proves \eqref{eq:chromatic-bound} without any restrictions on $G$.

\begin{theorem}\label{thm:main}
	Every finite simple graph $G$ satisfies
	\[
	\lcc(G)+\chi(G)\le |V(G)|+1.
	\]
\end{theorem}

The constant is sharp. Complete graphs of order at least two and nontrivial stars attain equality, as does $K_n-e$ for $n\ge3$. The proof of \Cref{thm:main} starts with an edge $uv$. Minimum clique partitions of the exclusive and common neighborhoods of its endpoints give a cover of all edges incident with $u$ or $v$, while adding at most one incidence at each remaining vertex. An independent set of nonneighbors of an endpoint improves its valency estimate by the size of that set. In an induced path $a-b-c-d$, each end edge has a common nonneighbor at the opposite end of the path. Inserting both end edges therefore fits the loss of two colors allowed by deleting the path. The remaining $P_4$-free case follows from the clique-number bound and Seinsche's theorem \cite{Seinsche1974}.

The equality problem requires finer control than this induction alone provides. Write
\[
r(G)=|V(G)|+1-\chi(G).
\]
A vertex is \emph{universal} if it is adjacent to every other vertex. Our second main result shows that the bound can be made strict at every nonuniversal vertex simultaneously.

\begin{theorem}\label{thm:localized}
	Every finite simple graph $G$ has an edge-clique cover $\mathcal C$ such that, for every $v\in V(G)$,
	\[
	\val_{\mathcal C}(v)\le
	\begin{cases}
		r(G),&v\text{ is universal},\\
		r(G)-1,&v\text{ is not universal}.
	\end{cases}
	\]
\end{theorem}

\begin{corollary}\label[corollary]{cor:universal}
	If $|V(G)|\ge2$ and $\lcc(G)+\chi(G)=|V(G)|+1$, then $G$ has a universal vertex. Equivalently, every graph without a universal vertex satisfies $\lcc(G)+\chi(G)\le |V(G)|$.
\end{corollary}

The stronger vertexwise bound requires structural information beyond the first induction. We argue by contradiction and show that a
counterexample of minimum order would be vertex-critical and double-critical. Join decompositions and a covering extension based on a maximum matching then force its complement to be connected, while a refined covering argument excludes nontrivial modules. These reductions yield an induced four-vertex path whose insertion improves the bound at every vertex. The double-critical
input is the established minimum-degree bound \cite{KawarabayashiPedersenToft2010}; the argument does not assume the Double-Critical Graph Conjecture.

Call a graph $G$ \emph{extremal} if $\lcc(G)+\chi(G)=|V(G)|+1$. The endpoint construction also yields structural restrictions beyond the existence of a universal vertex. If $G$ contains an induced matching of size $m\ge2$, then
\[
\lcc(G)+\chi(G)\le |V(G)|+3-m.
\]
Thus an extremal graph has no induced matching of size three. Moreover, deleting the four vertices of an induced $2K_2$ from an
extremal graph preserves extremality and lowers both the chromatic number and the local clique cover number by two.

The paper is organized as follows. \Cref{sec:prelim} introduces the notation and recalls the clique-number bound. \Cref{sec:extension} proves the endpoint estimate and \Cref{thm:main}, together with the induced matching consequences. \Cref{sec:universal-extensions} gives the matching-based extension used in the structural argument. \Cref{sec:localized} proves the stronger vertexwise bound in \Cref{thm:localized}. We conclude in \Cref{sec:conclusion} with the implications for the equality problem.

\section{Preliminaries}\label{sec:prelim}

All graphs are finite and simple. We write $N(v)$ and $N[v]$ for the open and closed neighborhoods of a vertex $v$. For $X\subseteq V(G)$, $G[X]$ denotes the induced subgraph, and $G-X=G[V(G)\setminus X]$. The symbols $\alpha(G)$, $\omega(G)$, $\chi(G)$, and $\nu(G)$ denote the independence number, clique number, chromatic number, and matching number. The join $F\vee H$ is obtained from the disjoint union $F\sqcup H$ by adding all edges between the two vertex sets. In particular, $K_s\vee H$ is obtained by adjoining $s$ universal vertices to $H$.

For an edge-clique cover $\mathcal C$, put
\[
\val_{\mathcal C}(v)=|\{C\in\mathcal C:v\in C\}|,
\qquad
\lcc(G)=\min_{\mathcal C}\max_{v\in V(G)}\val_{\mathcal C}(v).
\]
A cover is $q$-local if all its valencies are at most $q$. We sometimes regard covers as indexed families, so equal vertex sets may
occur more than once. Removing repetitions never increases a valency. Singleton cliques are permitted in vertex partitions and when extending a cover through a universal vertex; otherwise they can be discarded.

Write $\theta(F)=\chi(\overline F)$ for the minimum number of cliques in a vertex partition of $F$. Vertex clique covers give the same minimum, since their overlaps can be removed. We write $\varnothing$ for the graph with no vertices and adopt the conventions
$\chi(\varnothing)=\theta(\varnothing)=\lcc(\varnothing)=0$.
More generally, $\lcc(G)=0$ for every edgeless graph, since the
empty family of cliques covers all its edges. The parameter $\lcc$ is monotone under induced subgraphs: restrict each covering clique to the smaller vertex set and discard sets of order at most one. The Nordhaus--Gaddum inequality \cite{NordhausGaddum1956} gives
\begin{equation}\label{eq:NG}
	\theta(F)+\chi(F)\le |V(F)|+1
	\qquad(F\ne\varnothing).
\end{equation}

\begin{lemma}[Clique-number bound, \cite{BujtasDavoodiGyoriTuza2020}]\label[lemma]{lem:clique-bound}
	Every nonempty graph $G$ satisfies $\lcc(G)\le |V(G)|+1-\omega(G)$.
\end{lemma}

\begin{proof}
	Fix a maximum clique $Q$ and put $t=|V(G)\setminus Q|$. Use $Q$, the cliques $\{x\}\cup(N(x)\cap Q)$ for $x\notin Q$, and every edge of $G-Q$ individually, omitting cliques of order less than two. A vertex of $Q$ has valency at most $t+1$, and a vertex outside $Q$ has valency at most $1+(t-1)=t$.
\end{proof}

\section{Endpoint extensions and the chromatic inequality}\label{sec:extension}
We first construct a cover of all edges incident with either endpoint of a fixed edge, adding at most one incidence at each
other vertex. Independent sets outside the closed neighborhoods of the endpoints sharpen their valency bounds. Applying this construction to the two end edges of an induced $P_4$ gives the extension needed for the chromatic bound.

For an edge $uv$, partition $V(G)\setminus\{u,v\}$ into
\begin{equation}\label{eq:endpoint-cells}
	\begin{aligned}
		A&=N(u)\setminus N[v],& B&=N(v)\setminus N[u],\\
		D&=N(u)\cap N(v),& C&=V(G)\setminus(N[u]\cup N[v]).
	\end{aligned}
\end{equation}
Take minimum clique partitions $\mathcal P_A,\mathcal P_B,\mathcal P_D$ of the corresponding induced subgraphs. If $D\ne\varnothing$, let
\begin{equation}\label{eq:endpoint-family}
	\mathcal E_{uv}=
	\{Q\cup\{u\}:Q\in\mathcal P_A\}
	\cup\{Q\cup\{v\}:Q\in\mathcal P_B\}
	\cup\{Q\cup\{u,v\}:Q\in\mathcal P_D\}.
\end{equation}
If $D=\varnothing$, use the first two families and the clique $\{u,v\}$. The family is defined in the full graph $G$.

\begin{lemma}[Endpoint saving]\label[lemma]{lem:endpoints}
	The family $\mathcal E_{uv}$ covers all edges incident with $u$ or $v$. Every vertex in $A\cup B\cup D$ has valency one in this family, and every vertex in $C$ has valency zero. Moreover,
	\begin{equation}\label{eq:endpoint-saving}
		\val_{\mathcal E_{uv}}(u)\le r(G)-\alpha(G-N[u]),
		\qquad
		\val_{\mathcal E_{uv}}(v)\le r(G)-\alpha(G-N[v]).
	\end{equation}
	In particular, an independent set of $s$ common nonneighbors of $u,v$ reduces both endpoint bounds to $r(G)-s$.
\end{lemma}

\begin{proof}
	The coverage and the valencies away from the endpoints follow from the partitions. Put $a=|A|$, $b=|B|$, $c=|C|$, $d=|D|$, and
	\[
	h=\max\{1,\chi(G[A])\},\qquad j=\chi(G[D]).
	\]
	The valency at $u$ is $L_u=\theta(G[A])+\max\{1,\theta(G[D])\}$. Equation~\eqref{eq:NG}, with the empty set conventions, gives
	\begin{equation}\label{eq:empty-cells}
		\theta(G[A])\le a+1-h,
		\qquad \max\{1,\theta(G[D])\}\le d+1-j.
	\end{equation}
	Indeed, if $A$ is empty the first bound is $0\le0$, and if $D$ is empty the second is $1\le1$.
	
	Let $I\subseteq B\cup C=V(G)\setminus N[u]$ be independent. Color $A\cup\{v\}$ with $h$ colors, $D$ with $j$ disjoint colors,
	$(B\cup C)\setminus I$ individually, and $\{u\}\cup I$ with one further color. This gives
	$\chi(G)\le h+j+b+c-|I|+1$.
	Since $|V(G)|=a+b+c+d+2$, we obtain
	\[
	L_u\le a+d+2-h-j
	=|V(G)|-b-c-h-j
	\le r(G)-|I|.
	\]
	Maximize $|I|$ and then interchange $u,v$.
\end{proof}

\begin{lemma}[Induced-path extension]\label[lemma]{lem:path-extension}
	If $a,b,c,d$ induce the path $a-b-c-d$ and $H=G-\{a,b,c,d\}$, then
	\[
	\lcc(G)\le\max\{\lcc(H)+2,r(G)\}.
	\]
\end{lemma}

\begin{proof}
	Adjoin the families $\mathcal E_{ab}$ and $\mathcal E_{cd}$, both defined in $G$, to an optimal cover of $H$. The vertex $d$ is a common nonneighbor of $a,b$, and $a$ is a common nonneighbor of $c,d$. Thus every endpoint has valency at most $r(G)-1$ in its own endpoint family. The opposite family contributes zero at $a,d$ and one at $b,c$. Each vertex of $H$ receives at most
	two additional incidences. Every edge outside $H$ is incident with one of the endpoint pairs, so the resulting family covers $G$ with the asserted bound on its valency.
\end{proof}

\begin{proof}[Proof of \Cref{thm:main}]
	Induct on $n=|V(G)|$. The empty graph satisfies the assertion. If $G$ is nonempty and $P_4$-free, then $\chi(G)=\omega(G)$ by
	Seinsche's theorem \cite{Seinsche1974}, and \Cref{lem:clique-bound} applies. Otherwise choose an induced path $a-b-c-d$ and put
	$H=G-\{a,b,c,d\}$, $k=\chi(G)$, and $r=n+1-k$. Since the deleted path is bipartite, $\chi(H)\ge k-2$. By induction,
	\[
	\lcc(H)\le |V(H)|+1-\chi(H)
	\le n-3-(k-2)=r-2.
	\]
	The same bound holds if $H$ is empty. Now \Cref{lem:path-extension} gives $\lcc(G)\le r$.
\end{proof}

The endpoint estimate also yields a quantitative improvement when several edges can be inserted without interaction between their endpoints.

\begin{theorem}[Induced matchings]\label{thm:induced-matching}
	If $G$ has an induced matching $M$ of size $m\ge1$, then
	\begin{equation}\label{eq:matching-extension}
		\lcc(G)\le
		\max\{\lcc(G-V(M))+m,\ r(G)-m+1\}.
	\end{equation}
	Consequently, for $m\ge2$,
	\begin{equation}\label{eq:matching-improvement}
		\lcc(G)+\chi(G)\le |V(G)|+3-m.
	\end{equation}
\end{theorem}
\begin{proof}
	For each edge of $M$, choose one endpoint of every other edge. These $m-1$ vertices form an independent set of common nonneighbors of its endpoints. Adjoin all $m$ endpoint families to a cover of $G-V(M)$. Each residual vertex gains at most $m$ incidences; each endpoint has valency at most $r(G)-m+1$ and receives no incidence from the other endpoint families. This proves
	\eqref{eq:matching-extension}. The deleted graph is bipartite, so $\chi(G-V(M))\ge\chi(G)-2$. Applying \Cref{thm:main} to the residual graph bounds the first term by $r(G)+2-m$ and proves \eqref{eq:matching-improvement}.
\end{proof}

\begin{corollary}\label[corollary]{cor:matching-rigidity}
	Every extremal graph has induced matching number at most two. If an induced $2K_2$ in an extremal graph $G$ has vertex set $S$, then $H=G-S$ is extremal and
	\[
	\chi(H)=\chi(G)-2,\qquad \lcc(H)=\lcc(G)-2.
	\]
\end{corollary}
\begin{proof}
	The first assertion follows from \eqref{eq:matching-improvement}. For the second, put $r=r(G)$. Equation~\eqref{eq:matching-extension} requires $\lcc(H)\ge r-2$, while \Cref{thm:main} and $\chi(H)\ge\chi(G)-2$ give
	\[
	\lcc(H)\le |V(G)|-3-\chi(H)\le r-2.
	\]
	All these inequalities are equalities.
\end{proof}

\section{Cover extensions through universal vertices}\label{sec:universal-extensions}

The endpoint construction adds new covering cliques. When a universal vertex is adjoined, one may instead enlarge selected occurrences in an existing cover. The following lemma uses a maximum matching to select sufficiently few occurrences while preserving the valencies at nonisolated vertices. Preserving these valencies will allow us to retain the stricter bound at nonuniversal vertices when a universal vertex is restored in the proof of \Cref{thm:localized}.

\begin{lemma}[Matching extraction]\label[lemma]{lem:matching-extraction}
	If $H$ is nonempty of order $m$, then for every $s\ge1$,
	\[
	\lcc(K_s\vee H)\le\max\{\lcc(H),m-\nu(H)\}.
	\]
	More precisely, any edge-clique cover of $H$ can be extended through $K_s$ without changing the valencies of nonisolated vertices of $H$, with valency at most one at its isolated vertices and at most $m-\nu(H)$ at each added vertex.
\end{lemma}
\begin{proof}
	At each isolated vertex, discard any existing singleton occurrences and add exactly one singleton occurrence. For a maximum matching $M$, select one occurrence containing each edge of $M$ and one containing each unmatched vertex, counting an occurrence only once if it is selected more than once. At most
	\[
	|M|+(m-2|M|)=m-\nu(H)
	\]
	selected occurrences cover $V(H)$. Add all vertices of $K_s$ to these occurrences. The stated valency bounds follow. Since $m-\nu(H)\ge1$, the bound also absorbs the singleton valencies.
\end{proof}

\begin{lemma}\label[lemma]{lem:matching-degree}
	Let $d$ be a nonnegative integer. If $|V(F)|\ge2d+1$ and $\delta(F)\ge d$, then $\nu(F)\ge d$.
\end{lemma}
\begin{proof}
	Suppose a maximum matching $M$ has size $t<d$, and let $U$ be its unmatched set. Then $U$ is independent and $|U|\ge3$.
	For a matched edge $xy$, there cannot be distinct $u,w\in U$ with $ux,yw\in E(F)$, since $u-x-y-w$ would augment the matching. Thus at most $|U|$ incidences join $U$ to $\{x,y\}$: if both ends meet $U$, their unmatched neighbor sets are the same singleton and contribute two incidences. Consequently $d|U|\le e(U,V(M))\le t|U|$, a contradiction.
\end{proof}

\section{Localization of the maximum valency}\label{sec:localized}

To prove \Cref{thm:localized}, we refine the induced-path extension to meet the stricter bound at every nonuniversal vertex. This requires additional control both at the path vertices and at vertices that become universal after the path is deleted. We obtain this control by showing that a counterexample of minimum order is prime and double-critical and has connected complement.

A graph is \emph{vertex-critical} if deleting any vertex lowers its chromatic number. A connected graph is \emph{double-critical} if deleting the endpoints of every edge lowers its chromatic number by two. We use two standard results:
\begin{enumerate}[label=\textup{(\roman*)},leftmargin=2.6em]
	\item every noncomplete double-critical $h$-chromatic graph has minimum degree at least $h+1$ \cite[Proposition~3.9]{KawarabayashiPedersenToft2010};
	\item every vertex-critical $h$-chromatic graph with connected complement has at least $2h-1$ vertices, by Gallai's theorem \cite{Gallai1963,Stehlik2003}.
\end{enumerate}
All other reductions below follow from the extension estimates.

\subsection{Criticality and double-criticality}

Suppose that \Cref{thm:localized} fails, and choose a counterexample $G$ of minimum order. Throughout the remainder of this section, $G$ denotes this fixed counterexample, and we put
\[
n=|V(G)|,\qquad k=\chi(G),\qquad r=n+1-k.
\]
The empty graph and complete graphs satisfy the assertion, so $G$ is nonempty and noncomplete. Every proper induced subgraph has a cover with the vertexwise bounds specified in \Cref{thm:localized}.

The first reductions compare the change in the chromatic number with the incidences needed to restore one vertex or the endpoints
of an edge. The neighborhood partition below handles the former, and \Cref{lem:endpoints} handles the latter.

\begin{lemma}\label[lemma]{lem:neighborhood-partition}
	For every vertex $v$ of an arbitrary graph $F$,
	\[
	\theta(F[N(v)])\le
	\begin{cases}
		r(F),&v\text{ universal},\\
		r(F)-1,&v\text{ nonuniversal}.
	\end{cases}
	\]
\end{lemma}
\begin{proof}
	The universal case follows from \eqref{eq:NG} on $F-v$, with the empty case immediate. Otherwise the set $T$ of nonneighbors is nonempty. Color $N(v)$ optimally, color $T$ individually with fresh colors, and give $v$ one color used on $T$. If $N(v)\ne\varnothing$,
	\[
	\theta(F[N(v)])\le |N(v)|+1-\chi(F[N(v)])
	\le |N(v)|+1-\chi(F)+|T|=r(F)-1.
	\]
	The empty neighborhood case is immediate.
\end{proof}

\begin{lemma}\label[lemma]{lem:critical-reduction}
	If $G$ is a counterexample of minimum order to \Cref{thm:localized}, then $G$ is connected, vertex-critical, and double-critical, and it has no adjacent true twins.
\end{lemma}

\begin{proof}
	If $v$ is not chromatically critical, then $\chi(G-v)=k$ and $r(G-v)=r-1$. Take the asserted cover of $G-v$ and insert $v$ by adjoining it to the blocks of a minimum clique partition of $N(v)$. The valency at $v$ satisfies \Cref{lem:neighborhood-partition}. Every vertex nonuniversal in $G-v$ starts with valency at most $r-2$ and gains at most one incidence. A vertex universal in $G-v$ that receives an incidence is universal in $G$; one that misses $v$ receives none. This gives the required cover, a contradiction. Hence $G$ is vertex-critical and therefore connected.
	
	In a vertex-critical graph, nonadjacent vertices $x,y$ cannot satisfy $N(x)\subseteq N(y)$: a coloring of $G-x$ with $k-1$ colors could assign $x$ the color of $y$. Suppose that an edge $uv$ is not double-critical. Criticality gives $\chi(G-\{u,v\})=k-1$, so the target on $G-\{u,v\}$ is $r-1$. Insert $\mathcal E_{uv}$ into its asserted cover. The endpoint bounds follow from
	\Cref{lem:endpoints}; old nonuniversal vertices start at most $r-2$ and gain at most one incidence. If $x$ is universal in $G-\{u,v\}$ but not in $G$, then it misses both $u,v$: missing just $u$, say, would give $N(u)\subseteq N(x)$ because $uv$ is an edge. Thus such an $x$ gains no incidence. The remaining old universal vertices may have valency $r$, as permitted. This proves double-criticality.
	
	Finally suppose $u,v$ are adjacent true twins, that is, $N[u]=N[v]$. Since $r(G-u)=r$, add $u$ to every occurrence containing $v$ in the asserted cover of $G-u$. Connectedness and $G\ne K_2$ ensure that $v$ has a neighbor in $G-u$, so $uv$ is covered as well. All retained vertices keep their universality status under this deletion. Their valencies do not change, and the valency of $u$ equals that of $v$, which has the same universality status. This again gives the required cover, a contradiction.
\end{proof}

\subsection{Join decompositions and universal vertices}

We next exclude nontrivial join decompositions. The following two lemmas handle joins with two noncomplete factors, while the
extension from \Cref{sec:universal-extensions} handles the remaining case of a universal vertex. Excluding these cases makes Gallai's order bound available.

\begin{lemma}\label[lemma]{lem:strict-partition}
	If $J$ is noncomplete and double-critical, then
	\[
	\theta(J)\le |V(J)|-\chi(J)-1=r(J)-2.
	\]
\end{lemma}
\begin{proof}
	The double-critical degree bound and greedy coloring give
	\[
	\theta(J)=\chi(\overline J)
	\le\Delta(\overline J)+1
	=|V(J)|-\delta(J)
	\le |V(J)|-\chi(J)-1.
	\]
\end{proof}

\begin{lemma}\label[lemma]{lem:join}
	For nonempty graphs $X,Y$,
	\[
	\lcc(X\vee Y)\le
	\max\{\lcc(X)+\theta(Y),\lcc(Y)+\theta(X)\}.
	\]
\end{lemma}
\begin{proof}
	Take optimal edge-clique covers of $X,Y$, and add every clique $Q\cup R$ with $Q$ in a minimum clique partition of $X$ and $R$
	in one of $Y$. These product cliques cover the join edges and add $\theta(Y)$ incidences at each vertex of $X$ and $\theta(X)$ at each vertex of $Y$.
\end{proof}

\begin{proposition}\label[proposition]{prop:connected-complement}
	If $G$ is a counterexample of minimum order to \Cref{thm:localized}, then $\overline G$ is connected. In particular, $G$ has no universal vertex and $|V(G)|\ge2\chi(G)-1$.
\end{proposition}

\begin{proof}
	Suppose $\overline G$ is disconnected. Express $G$ as the join of its \emph{anticomponents}, the vertex sets of the components of
	$\overline G$. Chromatic additivity shows that each factor is vertex-critical. Every nonempty vertex-critical graph is connected, and a critical edgeless factor is a singleton. Double-criticality on internal edges therefore shows that each nonsingleton factor is double-critical. Groups of factors are likewise vertex-critical and double-critical, by the same additivity calculation.
	
	If at least two anticomponents are nonsingletons, partition the factors into two groups inducing noncomplete graphs $X,Y$. Since
	$r(X)+r(Y)=r+1$, \Cref{thm:main,lem:strict-partition,lem:join} give
	\[
	\lcc(G)\le r(X)+r(Y)-2=r-1,
	\]
	a contradiction. Otherwise $G=K_s\vee H$ with $H$ noncomplete and $\overline H$ connected. The absence of true twins gives $s=1$. Put $m=|V(H)|$ and $h=\chi(H)$. The graph $H$ is vertex-critical and noncomplete, so $h\ge2$ and $H$ has no isolated vertices. Since $\overline H$ is connected, $H$ has no universal vertex, and $r(H)=m+1-h=r$. 
	Minimality supplies a cover of $H$ with every valency at most $r-1$. Gallai's theorem gives $m\ge2h-1$, and $\delta(H)\ge h-1$. Thus \Cref{lem:matching-degree} gives $\nu(H)\ge h-1$. Apply the construction in \Cref{lem:matching-extraction} to this cover.
	It preserves the old vertex valencies and bounds the valency at the added universal vertex by $m-\nu(H)\le r$. This contradiction proves connectedness of $\overline G$. The final order bound is Gallai's theorem applied to $G$.
\end{proof}

\subsection{Modules and primeness}

Having excluded nontrivial joins, we next show that $G$ has no nontrivial module. If such a module existed, chromatic substitution
would transfer criticality and double-criticality to its induced subgraph. Combining covers of that subgraph and the remaining graph would then give an $(r-1)$-local cover of $G$.

A set $M\subseteq V(G)$ is a \emph{module} if every vertex outside $M$ is complete or anticomplete to $M$. It is nontrivial if $2\le |M|\le |V(G)|-1$, and a graph is \emph{prime} if it has no nontrivial module. Suppose $M$ is a module and write
\[
\begin{aligned}
	J&=G[M],\qquad F=G-M,\\
	A&=\{x\in V(F):x\text{ is complete to }M\},\\
	B&=\{x\in V(F):x\text{ is anticomplete to }M\}.
\end{aligned}
\]
For every integer $q\ge0$, let $F_q$ be obtained from $F$ by adding a $q$-clique complete to $A$ and anticomplete to $B$.

\begin{lemma}[Chromatic substitution]\label[lemma]{lem:substitution}
	With $M$, $J$, and $F_q$ as defined above, for every $X\subseteq M$, we have 
	\[
	\chi(G-X)=\chi\bigl(F_{\chi(J-X)}\bigr).
	\]
\end{lemma}

\begin{proof}
	Put $q=\chi(J-X)$. In a coloring of $F_q$, replace the added clique by a $q$-coloring of $J-X$ on the same palette. Conversely, the
	colors appearing on $J-X$ in a coloring of $G-X$ do not appear on $A$. Choose $q$ of these colors, recolor $J-X$ with that palette, and replace it by a $q$-clique. Vertices of $B$ impose no restriction in either direction.
\end{proof}

\begin{proposition}\label[proposition]{prop:prime}
	If $G$ is a counterexample of minimum order to \Cref{thm:localized}, then $G$ is prime.
\end{proposition}

\begin{proof}
	Suppose $M$ is nontrivial, with notation as above, and put $m=|M|$, $h=\chi(J)$. Connectedness of $G$ and $\overline G$ gives
	$A\ne\varnothing\ne B$. If $B=\{b\}$, then $N(b)\subseteq A\subseteq N(x)$ for every $x\in M$, contrary to criticality. Thus $|B|\ge2$.
	
	By \Cref{lem:substitution}, criticality of $G$ implies $\chi(J-x)=h-1$ for every $x\in M$, and hence $\chi(F_{h-1})=k-1$. Thus $J$ is vertex-critical and connected. If $xy\in E(J)$ and $\chi(J-\{x,y\})\ge h-1$, then \Cref{lem:substitution} and monotonicity of $\chi(F_q)$ in $q$ give $\chi(G-\{x,y\})\ge\chi(F_{h-1})=k-1$, contrary to double-criticality of $G$. Since deleting two vertices lowers the chromatic number by at most two, $J$ is double-critical. It is noncomplete, since a clique module would contain adjacent true twins. In particular $h\ge2$, and the double-critical degree bound gives $m\ge\delta(J)+1\ge h+2$. \Cref{lem:strict-partition} therefore gives $\theta(J)\le m-h-1$.
	
	Put $t=k-\chi(F)$. Criticality gives $t\ge1$, while using $h$ fresh colors on $J$ gives $t\le h$. Since $J$ and $G[B]$ are anticomplete and $h,|B|\ge2$,
	\[
	k\le\chi(G[A])+\max\{h,\chi(G[B])\}
	\le\chi(G[A])+h+|B|-2.
	\]
	Nordhaus--Gaddum on the nonempty graph $G[A]$ yields
	\begin{equation}\label{eq:module-saving}
		\theta(G[A])\le |A|+1-\chi(G[A])\le r-m+h-2.
	\end{equation}
	Take covers of $J$ and $F$ given by \Cref{thm:main}, and add all product cliques from minimum clique partitions of $J$ and $G[A]$.
	The resulting valencies are at most
	\[
	\begin{array}{rcll}
		(m+1-h)+(r-m+h-2)&=&r-1,&\quad\text{on }M,\\[1mm]
		(r-m+t)+(m-h-1)&\le&r-1,&\quad\text{on }A,\\[1mm]
		r-m+t&\le&r-2,&\quad\text{on }B.
	\end{array}
	\]
	Every edge between $M$ and $A$ is covered, and there are no edges between $M$ and $B$. This is an $(r-1)$-local cover of $G$, a contradiction.
\end{proof}

\subsection{A strict induced path extension}

We now return to the induced path extension. The first lemma allows a prescribed vertex of $G$ to occupy an internal position on an
induced $P_4$; choosing this vertex from an independent triple will provide the extra endpoint saving. The second lemma controls vertices that become universal after the path is deleted.

\begin{lemma}\label[lemma]{lem:internal-path}
	Let $F$ be a prime noncomplete vertex-critical graph. Every vertex of $F$ is internal to an induced $P_4$.
\end{lemma}

\begin{proof}
	Let $v$ be a vertex for which this fails. Each nonneighbor $y$ of $v$ has uniform adjacency to every anticomponent $M$ of the neighborhood graph $F[N(v)]$. Otherwise, along a complement path in $M$, there are consecutive vertices $x,x'$ such that $yx$ is an edge and $yx'$ is not; then $y-x-v-x'$ is an induced $P_4$. Consequently each such $M$ is a module of $F$: other anticomponents and $v$ are complete to it, and every nonneighbor of $v$ is uniform on it. Primeness forces each anticomponent to be a singleton. Thus $N(v)$ is a clique. Criticality gives $d(v)\ge\chi(F)-1$, so $N[v]$ contains a clique of order $\chi(F)$. A vertex-critical graph containing a clique of order $\chi(F)$ must be that clique itself, contrary to the hypothesis.
\end{proof}

\begin{lemma}\label[lemma]{lem:three-nonneighbors}
	Let $F$ be a double-critical graph. Every nonuniversal vertex of $F$ has at least three nonneighbors.
\end{lemma}

\begin{proof}
	Double-critical graphs are vertex-critical. A unique nonneighbor contradicts the non-domination property used in \Cref{lem:critical-reduction}. If a vertex $x$ has exactly two nonneighbors $y,z$, the same property forces $yz\in E(F)$.
	In a $(\chi(F)-2)$-coloring of $F-\{y,z\}$, the vertex $x$, which is universal in $F-\{y,z\}$, is the only vertex of its color. But every color in such a coloring must occur in $N(y)\cap N(z)$: if a color class contains no common neighbor, put its vertices adjacent to $y$ together with $z$ in one new color, and put its remaining vertices together with $y$ in the old color. This produces a $(\chi(F)-1)$-coloring. The singleton color of $x$ gives a contradiction.
\end{proof}

\begin{proof}[Proof of \Cref{thm:localized}]
	Continue with the smallest counterexample fixed above. By \Cref{prop:connected-complement,prop:prime}, it is prime, has no
	universal vertex, and satisfies $n\ge2k-1$. It has an independent triple: otherwise a $(k-2)$-coloring after deleting an edge would give $n-2\le2(k-2)$, contrary to this order bound. Choose a vertex $b$ of an independent triple. By \Cref{lem:internal-path}, there is an induced path $a-b-c-d$; moreover $\alpha(G-N[b])\ge2$.
	
	Put $H=G-\{a,b,c,d\}$. Since the deleted path is bipartite, $k\le\chi(H)+2$. Also, $H$ is an induced subgraph of $G-\{a,b\}$, so double-criticality of $ab$ gives $\chi(H)\le\chi(G-\{a,b\})=k-2$. Thus $\chi(H)=k-2$ and $r(H)=r-2$. Take the asserted cover of $H$ and adjoin $\mathcal E_{ab}$ and $\mathcal E_{cd}$, defined in $G$. In its own endpoint family, $b$ has valency at most $r-2$, while $a,c,d$ have valency at most $r-1$. Remove $c$ from its unique occurrence in $\mathcal E_{ab}$. This preserves coverage: every edge incident with $c$ is covered by $\mathcal E_{cd}$, and no other covered edge is lost. Discard a singleton residue if necessary. The vertices $a,d$ have no opposite family incidence, $b$ has at most one, and $c$ has none. Thus all four path vertices have valency at most $r-1$.
	
	Every nonuniversal vertex of $H$ starts with valency at most $r-3$ and gains at most two incidences. If $x$ is universal in $H$, it is nevertheless nonuniversal in $G$. Thus \Cref{lem:three-nonneighbors} gives at least three nonneighbors of $x$ in $G$, all on the deleted path. Hence $x$ has at most one neighbor on that path and gains at most one incidence. Its final valency is at most $r(H)+1=r-1$. The resulting cover is $(r-1)$-local, contradicting the choice of $G$. This proves \Cref{thm:localized} and \Cref{cor:universal}.
\end{proof}

\section{Concluding remarks}\label{sec:conclusion}

The endpoint estimate provides a common construction for the chromatic inequality and its induced matching strengthening. On an induced four-vertex path, the savings at the endpoints compensate for the additional covering incidences at the two middle vertices. This gives a short induction for the general bound. To obtain the simultaneous strict bound at nonuniversal vertices, the same construction is combined with criticality, join decompositions, and the elimination of nontrivial modules.

The stronger theorem identifies a necessary feature of every extremal graph: at least one vertex is universal. Consequently, the classification of equality cases reduces to graphs of the form $K_1\vee H$. The induced matching consequences impose further restrictions on such graphs, while \Cref{lem:matching-extraction} shows that a sufficiently large matching in $H$ can prevent equality. A complete characterization must determine when a cover of $H$ can be extended through the added vertex without reaching the chromatic bound. The proofs above identify both the local covering estimates and the structural restrictions available for that problem.

\bibliographystyle{plain}
\bibliography{Refs}

@article{BujtasDavoodiGyoriTuza2020,
	author  = {Csilla Bujt{\'a}s and Akbar Davoodi and Ervin Gy{\H{o}}ri and Zsolt Tuza},
	title   = {Clique coverings and claw-free graphs},
	journal = {European Journal of Combinatorics},
	volume  = {88},
	pages   = {103114},
	year    = {2020},
	doi     = {10.1016/j.ejc.2020.103114}
}

@article{DavoodiGerbnerMethukuVizer2020,
	author  = {Akbar Davoodi and D{\'a}niel Gerbner and Abhishek Methuku and
	M{\'a}t{\'e} Vizer},
	title   = {On clique coverings of complete multipartite graphs},
	journal = {Discrete Applied Mathematics},
	volume  = {276},
	pages   = {19--23},
	year    = {2020},
	doi     = {10.1016/j.dam.2019.09.014}
}

@article{DavoodiJavadiOmoomi2016PBD,
	author  = {Akbar Davoodi and Ramin Javadi and Behnaz Omoomi},
	title   = {Pairwise balanced designs and sigma clique partitions},
	journal = {Discrete Mathematics},
	volume  = {339},
	number  = {5},
	pages   = {1450--1458},
	year    = {2016},
	doi     = {10.1016/j.disc.2015.09.008}
}

@article{DavoodiJavadiOmoomi2016Sum,
	author  = {Akbar Davoodi and Ramin Javadi and Behnaz Omoomi},
	title   = {Edge Clique Covering Sum of Graphs},
	journal = {Acta Mathematica Hungarica},
	volume  = {149},
	number  = {1},
	pages   = {82--91},
	year    = {2016},
	doi     = {10.1007/s10474-016-0586-1}
}

@article{ErdosGoodmanPosa1966,
	author = {Paul Erd{\H{o}}s and Adolph Winkler Goodman and Louis P{\'o}sa},
	title   = {The Representation of a Graph by Set Intersections},
	journal = {Canadian Journal of Mathematics},
	volume  = {18},
	number  = {1},
	pages   = {106--112},
	year    = {1966},
	doi     = {10.4153/CJM-1966-014-3}
}

@article{Gallai1963,
	author  = {Tibor Gallai},
	title   = {Kritische {Graphen}. {II}},
	journal = {Magyar Tud. Akad. Mat. Kutat{\'o} Int. K{\"o}zl.},
	volume  = {8},
	number  = {3},
	pages   = {373--395},
	year    = {1963}
}

@article{JavadiMalekiOmoomi2016,
	author  = {Ramin Javadi and Zeinab Maleki and Behnaz Omoomi},
	title   = {Local Clique Covering of Claw-Free Graphs},
	journal = {Journal of Graph Theory},
	volume  = {81},
	number  = {1},
	pages   = {92--104},
	year    = {2016},
	doi     = {10.1002/jgt.21864}
}

@article{KawarabayashiPedersenToft2010,
	author  = {Kawarabayashi, Ken-ichi and Pedersen, Anders Sune and Toft, Bjarne},
	title   = {Double-Critical Graphs and Complete Minors},
	journal = {Electronic Journal of Combinatorics},
	volume  = {17},
	number  = {1},
	pages   = {R87},
	year    = {2010},
	doi     = {10.37236/359}
}

@article{NordhausGaddum1956,
	author = {Edward Alfred Nordhaus and Jerry William Gaddum},
	title   = {On Complementary Graphs},
	journal = {American Mathematical Monthly},
	volume  = {63},
	number  = {3},
	pages   = {175--177},
	year    = {1956},
	doi     = {10.2307/2306658}
}

@article{Seinsche1974,
	author  = {D. Seinsche},
	title   = {On a property of the class of {$n$}-colorable graphs},
	journal = {Journal of Combinatorial Theory, Series B},
	volume  = {16},
	number  = {2},
	pages   = {191--193},
	year    = {1974},
	doi     = {10.1016/0095-8956(74)90063-X}
}

@article{Stehlik2003,
	author  = {Mat\v{e}j Stehl{\'i}k},
	title   = {Critical graphs with connected complements},
	journal = {Journal of Combinatorial Theory, Series B},
	volume  = {89},
	number  = {2},
	pages   = {189--194},
	year    = {2003},
	doi     = {10.1016/S0095-8956(03)00069-8}
}

@article{SzpilrajnMarczewski1945,
	author  = {Edward Szpilrajn-Marczewski},
	title   = {Sur deux propri{\'e}t{\'e}s des classes d'ensembles},
	journal = {Fundamenta Mathematicae},
	volume  = {33},
	number  = {1},
	pages   = {303--307},
	year    = {1945},
	doi     = {10.4064/fm-33-1-303-307}
}

\end{document}